\documentclass[10pt]{amsart}
\usepackage{amssymb,amsthm,amsmath,amsfonts}
\usepackage{natbib}
\usepackage{hyperref}
\usepackage{enumerate}

\makeatletter
\g@addto@macro\th@plain{\thm@headpunct{}}
\makeatother

\newtheorem{theorem}{Theorem}[section]

\newtheorem{lemma}[theorem]{Lemma}

\newtheorem{corollary}[theorem]{Corollary}

\newcommand{\xx}{ {\textbf x} }
\newcommand{\ab}{ {\textbf a} }
\newcommand{\bb}{ {\textbf b} }
\newcommand{\cb}{ {\textbf c} }

\newcommand{\yy}{ {\textbf y} }
\newcommand{\ttt}{ {\textbf t} }
\newcommand{\sss}{ {\textbf s} }
\newcommand{\zz}{ {\textbf z} }
\newcommand{\ee}{ {\textbf e} }
\newcommand{\ub}{ {\textbf u} }

\newcommand{\VV}{ \Omega }
\newcommand{\RR}{\mathbb{R}}

\newcommand{\LL}{\mathbb{L}}
\newcommand{\PP}{\mathbb{P}}
\newcommand{\En}{\mathbb{E}}

\newcommand{\DD}{\mathcal{D}}

\newcommand{\diag}{\mathrm{diag}\,}
\newcommand{\Trace}{\mathrm{Trace}\,}
\newcommand{\DDet}{\mathrm{Det}}

\providecommand{\scalar}[1]{\left\langle#1\right\rangle}

\newcommand{\gw}{ {\mathfrak g} }
\newcommand{\ww}{ {\mathfrak w} }
\newcommand{\we}{ {\mathfrak w}_\ee }
\newcommand{\twe}{ \widetilde{\mathfrak w}_\ee }
\newcommand{\gwe}{ {\mathfrak g}_\ee }
\newcommand{\tgwe}{ \widetilde{\mathfrak g}_\ee }

\title[Fundamental equation of information]{The generalized fundamental equation of information on symmetric cones}
\author[B. Ko\l{}odziejek]{Bartosz Ko\l{}odziejek}
\address{Faculty of Mathematics and Information Science\\Warsaw University of Technology\\Pl. Politechniki 1\\00-661 Warszawa, Poland}
\email{kolodziejekb@mini.pw.edu.pl}


\begin{document}

\begin{abstract}
In this paper we generalize the fundamental equation of information to the symmetric cone domain and find general solution under the assumption of continuity of respective functions. 
 \end{abstract}
\maketitle

\section{Introduction}
The generalized fundamental equation of information (actually its specification with "constant multiplicative function") on the open vector domain is a functional equation of the form
\begin{align}\label{11dim}
F(x)+ G\left(\tfrac{y}{\underline{1}-x}\right)=H(y)+K\left(\tfrac{x}{\underline{1}-y}\right),
\end{align}
where $(x,y)\in \mathrm{D}_0=\left\{(x,y)\in(0,1)^r\times(0,1)^r\colon \forall i\in\{1,2,\ldots,r\}\,x_i+y_i\in(0,1)\right\}$ and $\underline{1}=(1,1,\ldots,1)\in\RR^r$. All operations in \eqref{11dim} are performed component-wise. First it was solved for $r=1$ by \cite{Mak1982} and later for any $r\in\mathbb{N}$ in \cite{EKN87}. There exists a vast literature regarding fundamental equation of information and its generalizations. History of main results together with references may be found in \cite{AczDar75,AczelNg83,Sand87}. 

Our aim is to analyze following generalization of \eqref{11dim}, when respective functions are defined on matrices. Let $\VV_+$ denote the cone of positive definite real symmetric matrices of rank $r$, $I$ is the identity matrix and let $\DD_+=\left\{\xx\in\VV_+\colon I-\xx\in\VV_+ \right\}$ be the analogue of $(0,1)$ interval in $\VV_+$. Consider unknown real functions $f$, $g$, $h$ and $k$ defined on $\DD_+$ that satisfy following functional equation
\begin{align}\label{2dim}
f(\xx)+g\left((I-\xx)^{-1/2}\cdot \yy\cdot (I-\xx)^{-1/2}\right)=h(\yy)+k\left((I-\yy)^{-1/2}\cdot\xx\cdot(I-\yy)^{-1/2}\right),
\end{align}
where $(\xx,\yy)\in\DD_0^+=\left\{(\xx,\yy)\in\DD_+^2\colon \xx+\yy\in\DD_+\right\}$ and $\cdot$ denotes the ordinary matrix product. Take $\xx=\ub\cdot\diag(x)\cdot\ub^T$ and $\yy=\ub\cdot\diag(y)\cdot\ub^T$, where $\ub$ is a fixed orthogonal matrix and $(x,y)\in\mathrm{D}_0$. Then $(\xx,\yy)\in\DD_0^+$, $\xx$ and $\yy$ commute and $(I-\xx)^{-1/2}=\ub\cdot\diag(\underline{1}-x)^{-1/2}\cdot \ub^T$. Thus, \eqref{2dim} gets the form
\begin{align}\label{3dim}
f_\ub(x)+ g_\ub\left(\tfrac{y}{\underline{1}-x}\right)=h_\ub(y)+k_\ub\left(\tfrac{x}{\underline{1}-y}\right)
\end{align}
for $(x,y)\in\mathrm{D}_0$, where $f_\ub(x)=f(\ub\cdot\diag(x)\cdot\ub^T)$, $g_\ub(x)=h(\ub\cdot\diag(x)\cdot\ub^T)$ and so on. This means that \eqref{2dim} is a generalization of \eqref{11dim} to a wider domain, since \eqref{3dim} is satisfied for any orthogonal matrix $\ub$ and when one takes non commutative $\xx$ and $\yy$, the situation is far more complicated, because the operations are not performed component-wise. It also justifies the name ``generalized fundamental equation of information on $\VV_+$'', despite its lack of clear connection to the developed information theory.

We want to go further and consider more general notion of division of matrices, which is defined through, so-called, multiplication algorithm. A multiplication algorithm is a mapping $w\colon\VV_+\mapsto GL(r,\RR)$ such that $w(\xx)\cdot w^T(\xx)=\xx$ for any $\xx\in\VV_+$. Multiplication algorithms (actually their inverses called division algorithms) were introduced by \cite{OlRu1962} alongside the characterization of Wishart probability distribution (see also \cite{CaLe1996} for generalization to symmetric cones). There exists infinite number of multiplication algorithms and the two basic examples of multiplication algorithms are $w_1(\xx)=\xx^{1/2}$ ($\xx^{1/2}$ being the unique positive definite symmetric square root of $\xx$) and $w_2(\xx)=t_\xx$, where $t_\xx$ is the lower triangular matrix from the Cholesky decomposition of $\xx=t_\xx\cdot t_\xx^T$. For unknown functions $f, g, h, k\colon \DD_+\to\RR$ we will analyze following functional equation
\begin{align}\label{Vplus}
f(\xx)+g\left(w(I-\xx)^{-1}\cdot \yy\cdot (w^T(I-\xx))^{-1}\right)=h(\yy)+k\left(\widetilde{w}(I-\yy)^{-1}\cdot\xx\cdot(\widetilde{w}^T(I-\yy))^{-1}\right),
\end{align}
where $(\xx,\yy)\in\DD_0^+$, $w$ and $\widetilde{w}$ are two multiplication algorithms satisfying additionally some natural properties. Note that for $(\xx,\yy)\in\DD_0^+$ the arguments of functions $g$ and $k$ belong to $\DD_+$.

Considering such generalization of \eqref{2dim}, in general, it is not possible to reduce \eqref{Vplus} to \eqref{11dim} as was possible for \eqref{2dim}. However, taking scalar matrices ($\xx=x I$ and $\yy=y I$ for $(x,y)\in (0,1)^2$ such that $x+y\in(0,1)$) equation \eqref{Vplus} comes down to \eqref{11dim} with $r=1$. This fact will be the crux in the proof of the main theorem.
The continuous solution to \eqref{Vplus} will be given in terms of, so-called, $w$-logarithmic Cauchy functions, i.e., functions that satisfy following functional equation 
\begin{align*}
f(\xx)+f(w(I)\cdot \yy\cdot w^T(I))=f(w(\xx)\cdot\yy\cdot w^T(\xx)),\quad (\xx,\yy)\in\VV_+.
\end{align*}
It is easy to see that $f(\xx)=H(\det\xx)$, where $H$ is a generalized logarithmic function ($H(ab)=H(a)+H(b)$, $a,b>0$), is always a $w$-logarithmic function for any $w$, but sometimes not the only one (see comment before Corollary \ref{cor3}). The general form of $w$-logarithmic Cauchy functions for two basic examples of multiplication algorithms $w_1(\xx)=\xx^{1/2}$ and $w_1(\xx)=t_\xx$ without any regularity assumptions were recently found in \cite{wC2013}. Later on we will write $\ww(\xx)\yy = w(\xx)\cdot \yy\cdot w^T(\xx)$, that is, in this case, $\ww(\xx)$ denotes the linear operator acting on $\VV_+$.

Finally, it must be noted that equation \eqref{Vplus} despite the lack of clear connection to the information theory is interesting to probabilists, since it is closely related to the characterization of matrix variate beta probability distribution. In \cite{WesBeta} and \cite{LajMe2009,LajMes12}, the problem of characterization of beta probability distribution was reduced to solving the fundamental equation of information with four unknown functions. Analogously, the solution to \eqref{Vplus} is used to characterize the matrix-variate beta probability distribution in \cite{BKBeta2014}.

All above considerations can be generalized to symmetric cones, of which $\VV_+$ is the prime example. The paper is organized as follows. In the next section we give necessary introduction to the theory of symmetric cones. Next, in Section 3.1 we recall known results concerning $w$-logarithmic Cauchy functions. Section 3.2 is devoted to the solution of \eqref{Vplus} in the symmetric cone setting.

\section{Preliminaries}
In this section, we recall basic facts of the theory of symmetric cones, which are needed in the paper. For further details, we refer to \citet{FaKo1994}. 

A \textit{Euclidean Jordan algebra} is a Euclidean space $\En$ (endowed with scalar product denoted $\scalar{\xx,\yy}$) equipped with a bilinear mapping (product)
\begin{align*}
\En\times\En \ni \left(\xx,\yy\right)\mapsto \xx\yy\in\En
\end{align*}
and a neutral element $\ee$ in $\En$ such that for all $\xx$, $\yy$, $\zz$ in $\En$:
\begin{enumerate}[(i)]
	\item $\xx\yy=\yy\xx$, 
	\item $\xx(\xx^2\yy)=\xx^2(\xx\yy)$,
	\item $\xx\ee=\xx$,
	\item $\scalar{\xx,\yy\zz}=\scalar{\xx\yy,\zz}$.
\end{enumerate}
For $\xx\in\En$ let $\LL(\xx)\colon \En\to\En$ be linear map defined by
\begin{align*}
\LL(\xx)\yy=\xx\yy,
\end{align*}
and define 
\begin{align*}
\PP(\xx)=2\LL^2(\xx)-\LL\left(\xx^2\right).
\end{align*} 
The map $\PP\colon \En\mapsto End(\En)$ is called the \emph{quadratic representation} of $\En$.

An element $\xx$ is said to be \emph{invertible} if there exists an element $\yy$ in $\En$ such that $\LL(\xx)\yy=\ee$. Then, $\yy$ is called the \emph{inverse of} $\xx$ and is denoted by $\yy=\xx^{-1}$. Note that the inverse of $\xx$ is unique. It can be shown that $\xx$ is invertible if and only if $\PP(\xx)$ is invertible and in this case $\left(\PP(\xx)\right)^{-1} =\PP\left(\xx^{-1}\right)$.

Euclidean Jordan algebra $\En$ is said to be \emph{simple} if it is not a \mbox{Cartesian} product of two Euclidean Jordan algebras of positive dimensions. Up to linear isomorphism there are only five kinds of Euclidean simple Jordan algebras. Let $\mathbb{K}$ denote either the real numbers $\RR$, the complex ones $\mathbb{C}$, quaternions $\mathbb{H}$ or the octonions $\mathbb{O}$, and write $S_r(\mathbb{K})$ for the space of $r\times r$ Hermitian matrices valued in $\mathbb{K}$, endowed with the Euclidean structure $\scalar{\xx,\yy}=\Trace(\xx\cdot\bar{\yy})$ and with the Jordan product
\begin{align}\label{defL}
\xx\yy=\tfrac{1}{2}(\xx\cdot\yy+\yy\cdot\xx),
\end{align}
where $\xx\cdot\yy$ denotes the ordinary product of matrices and $\bar{\yy}$ is the conjugate of $\yy$. Then, $S_r(\RR)$, $r\geq 1$, $S_r(\mathbb{C})$, $r\geq 2$, $S_r(\mathbb{H})$, $r\geq 2$, and the exceptional $S_3(\mathbb{O})$ are the first four kinds of Euclidean simple Jordan algebras. Note that in this case 
\begin{align}\label{defP}
\PP(\yy)\xx=\yy\cdot\xx\cdot\yy.
\end{align}
The fifth kind is the Euclidean space $\RR^{n+1}$, $n\geq 2$, with Jordan product
\begin{align}\label{scL}\begin{split}
\left(x_0,x_1,\dots, x_n\right)\left(y_0,y_1,\dots,y_n\right) =\left(\sum_{i=0}^n x_i y_i,x_0y_1+y_0x_1,\dots,x_0y_n+y_0x_n\right).
\end{split}
\end{align}

To each Euclidean simple Jordan algebra one can attach the set of Jordan squares
\begin{align*}
\bar{\VV}=\left\{\xx^2\colon\xx\in\En \right\}.
\end{align*}
The interior $\VV$ is a symmetric cone.
Moreover $\VV$ is \emph{irreducible}, i.e. it is not the Cartesian product of two convex cones. One can prove that an open convex cone is symmetric and irreducible if and only if it is the cone $\VV$ of some Euclidean simple Jordan algebra. Each simple Jordan algebra corresponds to a symmetric cone; hence, there exist up to linear isomorphism also only five kinds of symmetric cones. The cone corresponding to the Euclidean Jordan algebra $\RR^{n+1}$ equipped with Jordan product \eqref{scL} is called the Lorentz cone. Henceforth we will assume that $\VV$ is irreducible.

We denote by $G(\En)$ the subgroup of the linear group $GL(\En)$ of linear automorphisms which preserves $\VV$, and we denote by $G$ the connected component of $G(\En)$ containing the identity.  Recall that if $\En=S_r(\RR)$ and $GL(r,\RR)$ is the group of invertible $r\times r$ matrices, elements of $G(\En)$ are the maps $g\colon\En\to\En$ such that there exists $\ab\in GL(r,\RR)$ with
\begin{align*}
g(\xx)=\ab\cdot\xx\cdot\ab^T.
\end{align*}
We define $K=G\cap O(\En)$, where $O(\En)$ is the orthogonal group of $\En$. It can be shown that 
\begin{align*}
K=\{ k\in G\colon k\ee=\ee \}.
\end{align*} 

A \emph{multiplication algorithm} is a map $\VV\to G\colon \xx\mapsto \ww(\xx)$ such that $\ww(\xx)\ee=\xx$ for all $\xx\in\VV$. This concept is consistent with, so-called, division algorithm $\gw$, which was introduced by \citet{OlRu1962} and \citet{CaLe1996}, that is a mapping $\VV\ni\xx\mapsto \gw(\xx)\in G$ such that $\gw(\xx)\xx=\ee$ for any $\xx\in\VV$. If $\ww$ is a multiplication algorithm, then $\gw=\ww^{-1}$ is a division algorithm and vice versa; if $\gw$ is a division algorithm, then $\ww=\gw^{-1}$ is a multiplication algorithm. 
Note that $\VV$ is closed under multiplication $\xx\circ_\ww\yy=\ww(\xx)\yy$, but this multiplication is not commutative nor associative. It may also not have neutral element, but always $\ww(\ee)\in K$.
One of the two basic examples of multiplication algorithms is the map $\ww_1(\xx)=\PP\left(\xx^{1/2}\right)$. 

We will now introduce a very useful decomposition in $\En$, called \emph{spectral decomposition}. An element $\cb\in\En$ is said to be a \emph{idempotent} if $\cb\cb=\cb\neq 0$. Idempotents $\ab$ and $\bb$ are \emph{orthogonal} if $\ab\bb=0$. Idempotent $\cb$ is \emph{primitive} if $\cb$ is not a sum of two non-null idempotents. A \emph{complete system of primitive orthogonal idempotents} is a set $\left(\cb_1,\dots,\cb_r\right)$ such that
\begin{align*}
\sum_{i=1}^r \cb_i=\ee\quad\mbox{and}\quad\cb_i\cb_j=\delta_{ij}\cb_i\quad\mbox{for } 1\leq i\leq j\leq r.
\end{align*}
The size $r$ of such system is a constant called the \emph{rank} of $\En$. Any element $\xx$ of a Euclidean simple Jordan algebra can be written as $\xx=\sum_{i=1}^r\lambda_i\cb_i$ for some complete $\left(\cb_1,\dots,\cb_r\right)$ system of primitive orthogonal idempotents. The real numbers $\lambda_i$, $i=1,\dots,r$ are the \emph{eigenvalues} of $\xx$. An element $\xx\in\En$ belongs to $\VV$ if and only if all its eigenvalues are strictly positive. One can then define \emph{determinant} of $\xx$ by $\det\xx=\prod_{i=1}^r\lambda_i$. 

By \cite[Proposition III.4.3]{FaKo1994}, for any $g$ in the group $G$,
\begin{align*}
\det(g\xx)=(\DDet\,g)^{r/\dim\VV}\det\xx,
\end{align*}
where $\DDet$ denotes the determinant in the space of endomorphisms on $\VV$. Inserting a multiplication algorithm $g=\ww(\yy)$, $\yy\in\VV$, and $\xx=\ee$ we obtain
\begin{align*}
\DDet\left(\ww(\yy)\right) =(\det\yy)^{\dim\VV/r}
\end{align*}
and hence
\begin{align}\label{eqdetw}
\det(\ww(\yy)\xx) =\det\yy\det\xx
\end{align}
for any $\xx,\yy\in\VV$.

\section{Functional equations}\label{funeq}

\subsection{Logarithmic Cauchy functions}
Henceforth we will assume that $\VV$ is an irreducible symmetric cone.
As will be seen, the solution to fundamental equation of information will be given in terms of, so-called, $\ww$-logarithmic Cauchy functions, i.e., functions $f\colon\VV\to\RR$ that satisfy following functional equation 
\begin{align}\label{wC}
f(\xx)+f(\ww(\ee)\yy)=f(\ww(\xx)\yy),\quad (\xx,\yy)\in\VV^2,
\end{align}
where $\ww$ is a multiplication algorithm. Functional equation \eqref{wC} for $\ww_1(\xx)=\PP(\xx^{1/2})$ on $\VV_+$ was already considered in \cite{BW2003} for differentiable functions and in \cite{Molnar2006} for continuous functions on real or complex Hermitian positive definite matrices of rank greater than $2$. Without any regularity assumptions, it was solved on the Lorentz cone by \cite{Wes2007L}. Recently, the general form of $\ww_1$-logarithmic functions without any regularity assumptions was given in \cite{wC2013}. In this case $f(\xx)=H(\det\xx)$, where $H$ is generalized logarithmic function, i.e., $H(ab)=H(a)+H(b)$ for $a,b>0$ (see Theorem \ref{w1th}).

It should be stressed that there exists infinite number of multiplication algorithms. If $\ww$ is a multiplication algorithm, then trivial extensions are given by $\ww^{(k)}(\xx) = \ww(\xx)k$, where $k\in K$ is fixed. One may consider also multiplication algorithms of the form $P(\xx^{\alpha})t_{\xx^{1-2\alpha}}$, which interpolates between the two main examples: $w_1$ (which is $\alpha= 1/2$) and $w_2$ (which is $\alpha = 0$). In general, any multiplication algorithm may be written in the form $w(\xx) = \PP(\xx^{1/2})k_\xx$, where $k_\xx\in K$ and $K$ is the group of automorhisms.

Note that due to \eqref{eqdetw} function $H(\det\xx)$ is always a solution to \eqref{wC}, regardless of the choice of multiplication algorithm $\ww$, but may be not the only one - the best example is the multiplication algorithm related to the triangular group (see \cite[Theorem 3.5]{wC2013} and comment before Corollary \ref{cor3}). If a $\ww$-logarithmic functions $f$ is additionally $K$-invariant ($f(\xx)=f(k\xx)$ for any $k\in K$), then $H(\det\xx)$ is the only possible solution (Theorem \ref{XXX}). We now state the above mentioned results, which will be useful in the proof of the main theorem.

\begin{theorem}[$\ww_1$-logarithmic Cauchy functional equation]\label{w1th}
Let $f\colon \VV\to\RR$ be a function such that
\begin{align*}
f(\xx)+f(\yy)=f\left(\PP\left(\xx^{1/2}\right)\yy\right),\quad (\xx,\yy)\in\VV^2.
\end{align*}
Then there exists a generalized logarithmic function $H$ such that for any $\xx\in\VV$,
\begin{align*}
f(\xx)=H(\det\xx).
\end{align*}
\end{theorem}
\begin{theorem}\label{XXX}
Let $f\colon\VV\to\RR$ be a function satisfying \eqref{wC}. Assume additionally that $f$ is $K$-invariant, i.e., $f(k\xx)=f(\xx)$ for any $k\in K$ and $\xx\in\VV$. Then there exists a generalized logarithmic function $H$ such that for any $\xx\in\VV$,
$$f(\xx)=H(\det\xx).$$
\end{theorem}

Another results of \cite{wC2013} that we will need in the proof of the main theorem is the solution to Pexiderized version of \eqref{wC}:
\begin{lemma}[$\ww$-logarithmic Pexider functional equation]\label{lem1}
Assume that $a$, $b$, $c$ are real functions defined on the cone $\Omega$ satisfy following functional equation
\begin{align*}
a(\xx)+b(\yy)=c(\ww(\xx)\yy),\quad (\xx,\yy)\in\VV^2.
\end{align*}
Then there exist $\ww$-logarithmic function $f$ and real constants $a_0, b_0$ such that for any $\xx\in\VV$,
\begin{align*}
a(\xx) & =f(\xx)+a_0,\\
b(\xx) & =f(\ww(\ee)\xx)+b_0,\\
c(\xx) & =f(\xx)+a_0+b_0.
\end{align*}
\end{lemma}

\subsection{The fundamental equation of information with four unknown functions on symmetric cones}

Solution to the one-dimensional version of the fundamental equation with four unknown functions comes from \cite{Mak1982}. Independent, shorter, but with additional assumption of local integrability of functions was given by \cite{WesBeta} alongside the characterization of beta probability distribution. The problem when the equation is satisfied almost everywhere for measurable functions was considered in \cite{LajMe2009,LajMes12}. Recall the main result of \cite{Mak1982} (for $\alpha=0$ with the substitution of $H_1+H_3$ in place of $H_1$ compared to the original formulation):

\begin{theorem}[Gy. Maksa (1982)]\label{1dim}
Let $\mathrm{D}_0=\left\{(x,y)\in(0,1)^2\colon x+y\in(0,1)\right\}$ and assume that functions $F, G, H, K\colon(0,1)\to\RR$ satisfy \eqref{11dim} for any $(x,y)\in \mathrm{D}_0$. Then there exist generalized logarithmic functions $H_1, H_2, H_3$ and real constants $C_i$, $i=1,\ldots,4$, such that for any $x\in(0,1)$,
\begin{align*}
F(x) & = H_1(1-x)+H_2(x)+H_3(1-x)+C_1, \\
G(x) & = H_1(1-x)+H_3(x)+C_2, \\
H(x) & = H_1(1-x)+H_2(1-x)+H_3(x)+C_3, \\
K(x) & = H_1(1-x)+H_2(x)+C_4,
\end{align*}
and $C_1+C_2 = C_3+C_4$.
\end{theorem}

Henceforth we will assume that multiplication algorithms $\ww$ additionally satisfies following natural conditions
\begin{enumerate}[A.]
\item $\ww$ is homogeneous of degree $1$, that is $\ww(s\xx)=s\ww(\xx)$ for any $s>0$ and $\xx\in\VV$,
\item continuity in $\ee$, that is $\lim_{\xx\to\ee}\ww(\xx)=\ww(\ee)$,
\item surjectivity of the mapping $\VV\ni\xx\mapsto\gw(\xx)\ee\in\VV$.
\end{enumerate}
The same will be assumed for $\widetilde{\ww}$. By $\gw$ (resp. $\widetilde{\gw}$) we denote $\ww^{-1}$ (resp. $\widetilde{\ww}^{-1}$).
It is easy to construct a multiplication algorithm that does not satisfy $A$ and $B$, but we do not know whether there exists multiplication algorithm that does not satisfy condition $C$.

By $\we$ and $\gwe$ we will denote $\ww(\ee)$ and $\gw(\ee)$ respectively (analogously $\twe$ and $\tgwe$). Equation \eqref{Vplus} is rewritten to the symmetric cone setting in the following way:
\begin{align}\label{mainB}
f(\xx)+g(\gw(\ee-\xx)\yy)=h(\yy)+k(\widetilde{\gw}(\ee-\yy)\xx),\quad(\xx,\yy)\in\DD_0
\end{align}
for unknown functions $f, g, h, k\colon\DD\to\RR$, where $\DD=\left\{\xx\in\VV\colon\ee-\xx\in\VV\right\}$ and 
$$\DD_0:=\{(\ab,\bb)\in\DD^2\colon \ab+\bb\in\DD\}.$$ 
Note that $\DD_0$ is a subset of $\VV^2$, while $\mathrm{D}_0\subset\RR_+^2$. 

Analogous problem of considering general form of multiplication algorithm for Olkin-Baker functional equation was dealt with in \cite{BKLOR2014}, but that functional equation was much easier to solve. What is interesting, for scalar arguments, the fundamental equation of information can be brought to Olkin-Baker equation (see the proof of main theorem in \cite{Mak1982}). However it is not known if there exists such connection between these functional equations when one considers matrix or cone-variate arguments.

\begin{theorem}[Fundamental equation of information on symmetric cones]\label{inform}
Assume that \eqref{mainB} holds for continuous functions $f, g, h, k\colon\DD\to\RR$. If multiplication algorithms $\ww=\gw^{-1}$ and $\widetilde{\ww}=\widetilde{\gw}^{-1}$ satisfy conditions $A-C$, then there exist real constants $C_i$, $i=1,\ldots,4$, and continuous functions $h_i$, $i=1,2,3$, where 
\begin{itemize}
\item $h_1$ is $\ww$- and $\widetilde{\ww}-$logarithmic, 
\item $h_2$ is $\widetilde{\ww}$-logarithmic, 
\item $h_3$ is $\ww$-logarithmic function, 
\end{itemize}
such that for any $\xx\in\DD$,
\begin{align*}
f(\xx) & = h_1(\ee-\xx)+h_2(\xx)+h_3(\ee-\xx)+C_1, \\
g(\xx) & = h_1(\ee-\we\xx)+h_3(\we\xx)+C_2, \\
h(\xx) & = h_1(\ee-\xx)+h_2(\ee-\xx)+h_3(\xx)+C_3, \\
k(\xx) & = h_1(\ee-\twe\xx)+h_2(\twe\xx)+C_4,
\end{align*}
and $C_1+C_2 = C_3+C_4$.
\end{theorem}

The proof of Theorem \ref{inform} will be preceded by two lemmas. Note that in this lemmas it is additionally assumed that $\we=\twe=Id_\VV$, however, this does not affect the generality of Theorem \ref{inform}.

Let us observe that taking $\xx=\alpha\ee$ and $\yy=\beta\ee$ for $(\alpha,\beta)\in \mathrm{D}_0$, equation \eqref{mainB} is reduced to \eqref{11dim} with $F(\alpha)=f(\alpha\ee)$, $G(\alpha)=g(\alpha\ee)$, $H(\alpha)=h(\alpha\ee)$ and $K(\alpha)=k(\alpha\ee)$, hence by Theorem \ref{1dim} we know its general solution. We will use this fact several times in the proofs.

\begin{lemma}\label{lfconst}
Let $f\colon\DD\to\RR$ be a continuous function satisfying 
\begin{align}\label{fconst}
f(\ee-\xx)+f(\gw(\ee-\xx)\yy)=f(\yy)+f(\ee-\widetilde{\gw}(\ee-\yy)\xx),\quad(\xx,\yy)\in\DD_0.
\end{align}
If multiplication algorithms $\ww=\gw^{-1}$ and $\widetilde{\ww}=\widetilde{\gw}^{-1}$ satisfy conditions $A-C$, $\we=\twe=Id_\VV$, then $f$ is a constant.
\end{lemma}
\begin{proof}
In the first step let us observe that the function $(0,1)\ni\alpha\mapsto f(\alpha\ee)$ is constant. Indeed, it is an immediate consequence of Theorem \ref{1dim} for functions $F(\alpha)=K(\alpha)=f((1-\alpha)\ee)$ and $G(\alpha)=H(\alpha)=f(\alpha\ee)$. Without loss of generality we may assume that $f(\alpha\ee)=0$ for $\alpha\in(0,1)$.

Insert $\yy\mapsto\alpha\yy$ into \eqref{fconst} and take $\xx=\ee-\yy$. Then for $\alpha\in(0,1)$, after performing necessary rearrangements, we obtain
\begin{align}\label{fconst2}
f(\alpha\yy)-f(\alpha\ee)=f(\yy)-f(\ee-\widetilde{\gw}(\ee-\alpha\yy)(\ee-\yy)),\quad\yy\in\DD.
\end{align}
Because of condition $B$, we have (we have $\widetilde{\gw}(\ee)=Id_\VV$)
\begin{align*}
\lim_{\alpha\to0} \left\{\ee-\widetilde{\gw}(\ee-\alpha\yy)(\ee-\yy)\right\} = \yy\in\DD.
\end{align*}
Function $f$ is continuous on $\DD$, hence the right hand side of \eqref{fconst2} converges to $0$ as $\alpha\to0$. This implies that the limit of the left hand side of \eqref{fconst2} exists and likewise is equal to $0$. Using the fact that function $f(\alpha\ee)$ is zero, we get $\lim_{\alpha\to0} f(\alpha\yy)=0$ for $\yy\in\DD$. It follows that (just take $\yy=\beta\xx$ for $\beta>0$ small enough)
\begin{align}\label{fconst3}
\lim_{\alpha\to0} f(\alpha\xx)=0
\end{align}
for any $\xx\in\VV$.

Let us insert $\xx=\ee-\alpha(\ee+\ttt)$ and $\yy=\alpha\ee$ into \eqref{fconst} for $(\alpha,\alpha\ttt)\in(0,1)\times\DD$. Then, using condition $A$, we may write
\begin{align*}
f(\alpha(\ee+\ttt))+f(\gw(\ee+\ttt)\ee)=f(\alpha\ee)+f\left(\tfrac{\alpha}{1-\alpha}\ttt\right).
\end{align*}
Passing to the limit as $\alpha\to0$ in both sides of the above equality and using \eqref{fconst3}, we obtain for any $\ttt\in\VV$,
\begin{align*}
f(\gw(\ee+\ttt)\ee)=0.
\end{align*}
From this we will conclude that $f(\yy)=0$ for any $\yy\in\DD$.
Condition $C$ states that, for any $\yy\in\VV$ there exists $\xx\in\VV$ such that $\gw(\xx)\ee=\yy$. If $\yy\in\DD$ then $\xx\in\VV\setminus\DD$, because $\ee-\yy=\ee-\gw(\xx)\ee=\gw(\xx)(\xx-\ee)$ belongs to $\VV$ if and only if $\xx-\ee\in\VV$.
Therefore for any $\yy\in\DD$ there exists $\ttt\in\VV$ such that $\gw(\ee+\ttt)\ee=\yy$. Then
\begin{align*}
f(\yy)=f(\gw(\ee+\ttt)\ee)=0,
\end{align*}
what completes the proof.
\end{proof}

\begin{lemma}\label{H1}
Let $f,g\colon\DD\to\RR$ be continuous functions satisfying 
\begin{align}\label{H1eq}
f(\xx)+g(\gw(\ee-\xx)\yy)=g(\yy)+f(\widetilde{\gw}(\ee-\yy)\xx),\quad(\xx,\yy)\in\DD_0.
\end{align}
If multiplication algorithms $\ww=\gw^{-1}$ and $\widetilde{\ww}=\widetilde{\gw}^{-1}$ satisfy conditions $A-C$, $\we=\twe=Id_\VV$, then there exists $\ww$- and $\widetilde{\ww}$-logarithmic function $h$ such that for any $\xx\in\DD$,
\begin{align*}
f(\xx)=h(\ee-\xx)+f_0, \\
g(\xx)=h(\ee-\xx)+g_0,
\end{align*} 
for real constants $f_0$ and $g_0$.
\end{lemma}
\begin{proof}
As in the proof of Lemma \ref{lfconst}, let us consider \eqref{H1eq} for $\xx=\alpha\ee$ and $\yy=\beta\ee$, $(\alpha,\beta)\in \mathrm{D}_0$. Define $F(\alpha)=K(\alpha)=f(\alpha\ee)$ and $G(\alpha)=H(\alpha)=g(\alpha\ee)$. Theorem \ref{1dim} and continuity of $F,G,H, K$ imply that there exist real constants $\kappa$, $f_0$ and $g_0$ such that
\begin{align}\begin{split}\label{efg}
f(\alpha\ee)=\kappa\log(1-\alpha)+f_0,\\
g(\alpha\ee)=\kappa\log(1-\alpha)+g_0.
\end{split}\end{align} 
Without loss of generality we may assume $f_0=g_0=0$.

Let us take $(\xx,\yy)=(\sss,\ee-(\sss+\alpha\ttt))\in\DD_0$. It is easy to see that $(\xx,\yy)\in\DD_0$ if and only if  $(\sss,\alpha\ttt)\in\DD_0$. For $(\sss,\alpha\ttt)\in\DD_0$ equation \eqref{H1eq} gets the form (after rearrangements)
\begin{align*}
g(\ee-(\sss+\alpha\ttt)))-f(\sss)=g(\ee-\alpha\gw(\ee-\sss)\ttt)-f(\widetilde{\gw}(\sss+\alpha\ttt)\sss).
\end{align*}
The limit as $\alpha\to0$ of the left hand side of the above equality exists, hence the limit of right hand side exists as well. Passing to the limit as $\alpha\to0$, we obtain
\begin{align}\label{varat}
g(\ee-\sss)-f(\sss)=\lim_{\alpha\to0}\left\{g(\ee-\alpha\gw(\ee-\sss)\ttt)-f(\widetilde{\gw}(\sss+\alpha\ttt)\sss)\right\}.
\end{align}
Define
\begin{align*}
h(\xx):=\lim_{\alpha\to0}\left\{g(\ee-\alpha\xx)-\kappa\log\alpha\right\}.
\end{align*}
We will show that $h(\xx)$ exists for any $\xx\in\VV$.
Inserting $\sss=\ttt\in\DD$ in \eqref{varat}, we get
\begin{align*}
g(\ee-\ttt)-f(\ttt)=\lim_{\alpha\to0}\left\{g(\ee-\alpha\gw(\ee-\ttt)\ttt)-f\left(\tfrac{1}{\alpha+1}\ee\right)\right\}.
\end{align*}
Recall that $f\left(\tfrac{1}{\alpha+1}\ee\right)=\kappa\log\tfrac{\alpha}{\alpha+1}$. Therefore we obtain for any $\ttt\in\DD$,
\begin{align}\label{varat2}
g(\ee-\ttt)-f(\ttt)=h(\gw(\ee-\ttt)\ttt).
\end{align}
Because of condition $C$, for any $\yy\in\VV$ there exists $\xx\in\VV$ such that $\gw(\xx)\ee=\yy$. In the proof of previous lemma it was shown that if $\yy\in\DD$, then $\xx\in\VV\setminus\DD$. Analogously, if $\yy\in\VV\setminus\DD$, then $\xx\in\DD$.
Take $\yy\in\VV$, then $(\ee+\yy)\in\VV\setminus\DD$, hence there exists element $\xx=\ee-\ttt\in\DD$ such that $\gw(\ee-\ttt)\ee=\yy+\ee$, i.e., $\yy=\gw(\ee-\ttt)\ee-\ee=\gw(\ee-\ttt)\ttt$.
This implies that function $h$ is well defined on the whole $\VV$. 

Using the definition of $h$ and the form of the one-dimensional solutions \eqref{efg}, we obtain that $h(\beta\ee)=\kappa\log\beta$ and
\begin{align}\label{hbeta}
\begin{split}
h(\beta\xx)& =\lim_{\alpha\to0}\left\{g(\ee-\alpha\beta\xx)-\kappa\log\alpha\right\} \\
&=\lim_{\alpha\to0}\left\{g(\ee-\alpha\beta\xx)-\kappa\log\alpha\beta\right\}+\kappa\log\beta\\
& = h(\xx)+h(\beta\ee)
\end{split}\end{align}
for any $\beta>0$ and $\xx\in\VV$.
Returning to \eqref{varat}, using \eqref{varat2} and the definition of $h$, we get
\begin{align}\label{varat3}
h(\gw(\ee-\sss)\sss)=h(\gw(\ee-\sss)\ttt)-\lim_{\alpha\to0}\left\{f(\widetilde{\gw}(\sss+\alpha\ttt)\sss)-\kappa\log\alpha\right\}.
\end{align}
Note that due to \eqref{hbeta}, above equation holds for any $\ttt\in\VV$.
Our aim is to show that $h$ is $\ww$-logarithmic. Put $\sss\mapsto\beta\sss$ in \eqref{varat3} for $\sss\in\DD$ and $\beta>0$ and, by homogeneity of $\ww$, we have $\gw(\beta\sss+\alpha\ttt)\beta\sss=\gw(\sss+\frac{\alpha}{\beta}\ttt)\sss$. Thus, using \eqref{hbeta} we arrive at
\begin{align*}
h(\gw(\ee-\beta\sss)\sss)=h(\gw(\ee-\beta\sss)\ttt)-\lim_{\alpha\to0}\left\{f(\gw(\sss+\frac{\alpha}{\beta}\ttt)\sss)-\kappa\log\frac{\alpha}{\beta}\right\}.
\end{align*}
Observe that, the limit on the right hand side of the above equation does not depend on $\beta$. Therefore, we may pass to the limit as $\beta\to0$ to obtain ($h$ is continuous and $\gwe=Id_\VV$)
\begin{align*}
h(\sss)=h(\ttt)-\lim_{\alpha\to0}\left\{f(\widetilde{\gw}(\sss+\alpha\ttt)\sss)-\kappa\log\beta\right\}.
\end{align*}
Now, subtracting above equation from \eqref{varat3} we get
\begin{align*}
h(\gw(\ee-\sss)\sss)-h(\sss)=h(\gw(\ee-\sss)\ttt)-h(\ttt)
\end{align*}
for any $\sss\in\DD$ and $\ttt\in\VV$. Put $\sss=\ee-\alpha\xx$, $\ttt=\ww(\xx)\yy$ and use the homogeneity of $\ww$ and property \eqref{hbeta} to obtain (after rearrangements)
\begin{align*}
h(\ee-\alpha\xx)-h(\gw(\xx)(\ee-\alpha\xx))+h(\yy)=h(\ww(\xx)\yy)
\end{align*}
Again, passing to the limit as $\alpha\to0$ we obtain (recall that $h(\ee)=\kappa\log1=0$)
\begin{align}
-h(\gw(\xx)\ee)+h(\yy)=h(\ww(\xx)\yy),\quad(\xx,\yy)\in\VV^2
\end{align}
which is the $\ww$-logarithmic Pexider equation, thus Lemma \ref{lem1} implies the existence of $\ww$-logarithmic function $F$ such that
$h(\xx)=F(\xx)+a_0$ and $-h(\gw(\xx)\ee)=F(\xx)+b_0$. Since $h(\ee)=0$, we have $a_0=b_0=0$ (put $\xx=\ee$).

Summing up, we have shown that $h(\xx):=\lim_{\alpha\to0}\left\{g(\ee-\alpha\xx)-\kappa\log\alpha\right\}$ is $\ww$-logarithmic and that \eqref{varat2} holds.
By the symmetry of \eqref{H1eq} we conclude that function $\widetilde{h}(\xx):=\lim_{\alpha\to0}\left\{f(\ee-\alpha\xx)-\kappa\log\alpha\right\}$ is $\widetilde{\ww}$-logarithmic and
\begin{align}\label{symm}
f(\ee-\ttt)-g(\ttt)=\widetilde{h}(\widetilde{\gw}(\ee-\ttt)\ttt) ,\quad\ttt\in\DD.
\end{align}
Therefore, by \eqref{varat2} and \eqref{symm} we have
\begin{align*}
h(\gw(\ee-\ttt)\ttt)=g(\ee-\ttt)-f(\ttt)=-\widetilde{h}(\widetilde{\gw}(\ttt)(\ee-\ttt)) ,\quad\ttt\in\DD.
\end{align*}
Using \eqref{wC} on both sides we obtain 
\begin{align*}
h(\ttt)-h(\ee-\ttt)=\widetilde{h}(\ttt)-\widetilde{h}(\ee-\ttt),
\end{align*}
which by standard argument (take $\ttt=\alpha\xx$ and pass to the limit as $\alpha\to0$) implies that $h=\widetilde{h}$, that is, $h$ is both $\ww$- and $\widetilde{\ww}$-logarithmic.

Define now $\bar{g}(\xx)=g(\xx)-h(\ee-\xx)$, $\xx\in\DD$. Then by \eqref{varat2} we have
\begin{align}\label{deff2}
f(\xx)&=\bar{g}(\ee-\xx)+h(\ee-\xx).
\end{align}
We will show that $\bar{g}$ is constant. Using \eqref{deff2} and the definition of $\bar{g}$ to eliminate functions $f$ and $g$ in the main equation \eqref{H1eq}, we get for $(\xx,\yy)\in\DD_0$:
\begin{multline*}
\left(\bar{g}(\ee-\xx)+h(\ee-\xx)\right)+\left(\bar{g}(\gw(\ee-\xx)\yy)+h(\ee-\gw(\ee-\xx)\yy) \right) \\
= \left(\bar{g}(\yy)+h(\ee-\yy) \right)+\left(\bar{g}(\ee-\widetilde{\gw}(\ee-\yy)\xx)+h(\ee-\widetilde{\gw}(\ee-\yy)\xx)\right).
\end{multline*}
Since $\ee-\gw(\ee-\yy)\xx=\gw(\ee-\yy)(\ee-\xx-\yy)$ and using the fact that $h$ is $\ww$- and $\widetilde{\ww}$-logarithmic (that is $h(\gw(\ab)\bb)=h(\bb)-h(\ab)=h(\widetilde{\gw}(\ab)\bb)$, $\ab,\bb\in\VV$) we obtain 
\begin{align*}
\bar{g}(\ee-\xx)+\bar{g}(\gw(\ee-\xx)\yy)=\bar{g}(\yy)+\bar{g}(\ee-\widetilde{\gw}(\ee-\yy)\xx),\quad (\xx,\yy)\in\DD_0.
\end{align*}
Function $\bar{g}$ is continuous, therefore Lemma \ref{1dim} implies that $\bar{g}$ is constant what completes the proof.
\end{proof}

Now we are ready to prove the main theorem.
\begin{proof}[ Proof of Theorem~\ref{inform}]
First observe that without loss of generality we may assume that $\gwe=Id_\VV$. Indeed, mapping $\hat{\gw}\colon\xx\mapsto \we\gw(\xx)$ defines division algorithm with this property, since $\hat{\gw}(\ee)=\we\gwe=Id_\VV$. Thus, for $\hat{g}(\xx):=g(\gwe\xx)$ we have $g(\gw(\ee-\xx)\yy)=g(\gwe\we\gw(\ee-\xx)\yy)=\hat{g}(\hat{\gw}(\ee-\xx)\yy)$. The same can be done for $\widetilde{\gw}$ and function $k$. In this way we arrive at the equation \eqref{mainB} with division algorithms $\gw$ and $\widetilde{\gw}$ (and also functions $g$ and $k$) redefined in a way that $\gwe=\tgwe=Id_\VV$ (compare the formulation of Theorem~\ref{inform}).

Consider equation \eqref{mainB} for $(\alpha\xx,\ee-\yy)\in\DD$ and pass to the limit as $\alpha\to0$. By continuity of $\gw$ in $\ee$ and $g$ on $\DD$ we get
\begin{align}\label{A1}
\lim_{\alpha\to0}\left\{f(\alpha\xx)-k(\alpha \widetilde{\gw}(\yy)\xx)\right\}=h(\ee-\yy)-g(\ee-\yy),\quad (\xx,\yy)\in\VV\times\DD.
\end{align}
Insert $\yy=\xx\in\DD$ into \eqref{A1}, and get
\begin{align*}
\lim_{\alpha\to0}\left\{f(\alpha\xx)-k(\alpha\ee)\right\}=h(\ee-\xx)-g(\ee-\xx),\quad \xx\in\DD.
\end{align*}
Therefore the limit
\begin{align*}
l_1(\xx):=\lim_{\alpha\to0}\left\{f(\alpha\xx)-k(\alpha\ee)\right\}
\end{align*} 
exists for any $\xx\in\DD$ and
\begin{align}\label{l1eq}
l_1(\xx)=h(\ee-\xx)-g(\ee-\xx),\quad\xx\in\DD.
\end{align}
We will prove that this limit exists for any $\xx\in\VV$. Taking as in the proofs of previous lemmas $\xx=\alpha\ee$ and $\yy=\beta\ee$ in \eqref{mainB} for $(\alpha,\beta)\in \mathrm{D}_0$, and defining $F(\alpha)=f(\alpha\ee)$, $G(\alpha)=g(\alpha\ee)$, $H(\alpha)=h(\alpha\ee)$ and $K(\alpha)=k(\alpha\ee)$, Theorem \ref{1dim} implies that there exist constants $\kappa_i$, $i=1,2,3$, and $C_1$, $C_4$ such that for $\alpha\in(0,1)$,
\begin{align*}
F(\alpha)=f(\alpha\ee) & = \kappa_1\log(1-\alpha)+\kappa_2\log\alpha+\kappa_3\log(1-\alpha)+C_1, \\
K(\alpha)=k(\alpha\ee) & = \kappa_1\log(1-\alpha)+\kappa_2\log\alpha+C_4,
\end{align*}
holds. Without loss of generality we may assume $C_1=C_4=0$. Therefore, we have
\begin{align*}
& \lim_{\alpha\to0}\left\{k(\alpha\ee)-\kappa_2\log\alpha\right\} =0,\\
l_1(\beta\ee)= & \lim_{\alpha\to0}\left\{f(\alpha\beta\ee)-k(\alpha\ee)\right\}=\kappa_2\log\beta.
\end{align*}
Above equalities and similar considerations as in \eqref{hbeta} imply that for $\xx\in\DD$ and $\beta>0$,
\begin{align}\label{l1}
l_1(\xx)+l_1(\beta\ee)=l_1(\beta\xx).
\end{align}
The limit $l_1(\xx)$ therefore exists for any $\xx\in\VV$.
 
Returning to \eqref{A1}, we get
\begin{align*}
l_1(\xx)-\lim_{\alpha\to0}\left\{k(\alpha \widetilde{\gw}(\yy)\xx)-k(\alpha\ee)\right\}=h(\ee-\yy)-g(\ee-\yy),\quad (\xx,\yy)\in\VV\times\DD.
\end{align*}
Defining $l_2(\xx)=\lim_{\alpha\to0}\left\{k(\alpha \xx)-k(\alpha\ee)\right\}$ and using \eqref{l1eq} we get
\begin{align*}
l_1(\xx)-l_2(\widetilde{\gw}(\yy)\xx)=l_1(\yy),\quad (\xx,\yy)\in\VV\times\DD,
\end{align*}
that is, for $\xx=\widetilde{\ww}(\yy)\zz$ we obtain
\begin{align}\label{A2}
l_1(\widetilde{\ww}(\yy)\zz)=l_2(\zz)+l_1(\yy),\quad (\zz,\yy)\in\VV\times\DD.
\end{align}
Due to property \eqref{l1} the above equation holds for any $(\zz,\yy)\in\VV^2$: just take $\yy=\beta\xx$ for $\beta>0$ small enough. This is the $\widetilde{\ww}$-logarithmic Pexider equation, thus Lemma \ref{lem1} implies the existence of $\widetilde{\ww}$-logarithmic function $h_2$ and constant $\gamma$ such that
\begin{align*}
l_1(\xx)=h_2(\xx)+\gamma, \quad\xx\in\VV.
\end{align*}
Due to \eqref{l1} we obtain $\gamma=0$. Hence, \eqref{l1eq} implies that for any $\xx\in\DD$, 
\begin{align}\label{h3}
h(\xx)=h_2(\ee-\xx)+g(\xx).
\end{align}

By the symmetry of \eqref{mainB} we conclude that there exists a $\ww$-logarithmic function $h_3$ such that
\begin{align}\label{h2}
f(\xx)=h_3(\ee-\xx)+k(\xx).
\end{align} 
Using \eqref{h3} and \eqref{h2} to elimination of functions $f$ and $h$ in \eqref{mainB}, we get
for $(\xx,\yy)\in\DD_0$,
\begin{align}\label{almfin}
h_3(\ee-\xx)+k(\xx)+g(\gw(\ee-\xx)\yy)=h_2(\ee-\yy)+g(\yy)+k(\widetilde{\gw}(\ee-\yy)\xx).
\end{align}
Let us define for $\xx\in\VV$,
\begin{align*}
\bar{g}(\xx)=g(\xx)-h_3(\xx),\\
\bar{k}(\xx)=k(\xx)-h_2(\xx).
\end{align*}
Exploiting above equations in \eqref{almfin} we have for $(\xx,\yy)\in\DD_0$,
\begin{multline*}
h_3(\ee-\xx)+\left(\bar{k}(\xx)+h_2(\xx)\right)
+\left(\bar{g}(\gw(\ee-\xx)\yy)+h_3(\gw(\ee-\xx)\yy) \right)\\
=h_2(\ee-\yy)+\left(\bar{g}(\yy)+h_3(\yy)\right)
+\left(\bar{k}(\gw(\ee-\yy)\xx)+h_2(\gw(\ee-\yy)\xx )\right).
\end{multline*}
Since $\ee-\gw(\ee-\xx)\yy=\gw(\ee-\xx)(\ee-\xx-\yy)$ and $h_2$ and $h_3$ are $\widetilde{\ww}$ and $\ww$-logarithmic respectively ($h_2(\widetilde{\gw}(\ab)\bb)=h_2(\bb)-h_2(\ab)$ and $h_3(\gw(\ab)\bb)=h_3(\bb)-h_3(\ab)$ for $\ab,\bb\in\VV$), 
functions $\bar{g}$ and $\bar{k}$ satisfy following functional equation
\begin{align}
\bar{k}(\xx)+\bar{g}(\gw(\ee-\xx)\yy)=\bar{g}(\yy)+\bar{k}(\widetilde{\gw}(\ee-\yy)\xx),\quad(\xx,\yy)\in\DD_0.
\end{align}
Since $\bar{g}$ and $\bar{k}$ are continuous, Lemma \ref{H1} implies existence of continuous $\ww$- and $\widetilde{\ww}$-logarithmic function $h_1$ and real constants $g_0$, $k_0$ such that for $\xx\in\VV$,
\begin{align*}
\bar{g}(\xx)& =h_1(\ee-\xx)+g_0,\\
\bar{k}(\xx)& =h_1(\ee-\xx)+k_0,
\end{align*}
thus,
\begin{align*}
g(\xx) & =\bar{g}(\xx)+h_3(\xx)=h_1(\ee-\xx)+h_3(\xx)+g_0,\\
k(\xx) & =\bar{k}(\xx)+h_2(\xx)=h_1(\ee-\xx)+h_2(\xx)+k_0.
\end{align*}
Comparing above equality with expression for $k(\alpha\ee)$ we see that $k_0=C_4=0$.
By \eqref{h3} and \eqref{h2} we have
\begin{align*}
f(\xx) & =h_3(\ee-\xx)+k(\xx)=h_1(\ee-\xx)+h_2(\xx)+h_3(\ee-\xx),\\
h(\xx) & =h_2(\ee-\xx)+g(\xx)=h_1(\ee-\xx)+h_2(\ee-\xx)+h_3(\xx)+g_0,
\end{align*}
what completes the proof.
\end{proof}

Having solved fundamental equation of information for any multiplication algorithm it is easy to give the specification:
\begin{corollary}\label{cor1}
Assume that $f$, $g$, $h$ and $k$ are continuous real functions defined on $\DD$ and
\begin{align*}
f(\xx)+g(\gw_1(\ee-\xx)\yy)=h(\yy)+k(\gw_1(\ee-\yy)\xx),\quad(\xx,\yy)\in\DD_0,
\end{align*}
where $\gw_1(\xx)=\ww_1(\xx)^{-1}=\PP\left(\xx^{-1/2}\right)$. Then there exist constants $\kappa_j$, $j=1,2,3$ and $C_i$, $i=1,\ldots,4$, such that for any $\xx\in\DD$,
\begin{align}\begin{split}\label{tryk}
f(\xx) & = \kappa_1\log\det(\ee-\xx)+\kappa_2\log\det\xx+\kappa_3\log\det(\ee-\xx)+C_1, \\
g(\xx) & = \kappa_1\log\det(\ee-\xx)+\kappa_3\log\det\xx+C_2, \\
h(\xx) & = \kappa_1\log\det(\ee-\xx)+\kappa_2\log\det(\ee-\xx)+\kappa_3\log\det\xx+C_3, \\
k(\xx) & = \kappa_1\log\det(\ee-\xx)+\kappa_2\log\det\xx+C_4,
\end{split}\end{align}
and $C_1+C_2 = C_3+C_4$.
\end{corollary}
The proof quickly follows from Theorem \ref{w1th} and Theorem \ref{inform}.

In order to show an example of a multiplication algorithm for which the solution to generalized fundamental equation of information on $\VV$ is of different form than given in \eqref{tryk}, we should give full introduction regarding triangular group on $\VV$ and related Gauss decomposition. Plan of this paper was to keep it as concise as possible, therefore we decided to restrict to the cone $\VV_+$ of positive definite real symmetric matrices of rank $r$, where the Gauss decomposition on $\VV_+$ and the Cholesky decomposition coincide. We adhere to the notation from Introduction, $t_\xx$ is the lower triangular matrix from the Cholesky decomposition of $\xx=t_\xx\cdot t_\xx^T$. Then it is easy to see that $w_2(\xx)=t_\xx$ is a multiplication algorithm. Let $\Delta_k(\xx)$ denote the $k$th principal minor of $\xx$ and define for $s\in\RR^r$ the \emph{generalized power function}
$$\Delta_s(\xx)=\Delta_1(\xx)^{s_1-s_2}\Delta_2(\xx)^{s_2-s_3}\ldots\Delta_r(\xx)^{s_r}.$$
Then $\log\Delta_s(\xx)$ is the general form of a continuous $w_2$-logarithmic function (see \citet[Theorem 3.5]{wC2013}) and we have the following
\begin{corollary}\label{cor3}
Assume that $f$, $g$, $h$ and $k$ are continuous real functions defined on $\DD_+$ and
\begin{align*}
f(\xx)+g(t_{I-\xx}\cdot\yy\cdot t_{I-\xx}^T)=h(\yy)+k(t_{I-\yy}\cdot\xx\cdot t_{I-\yy}^T),\quad(\xx,\yy)\in\DD^+_0.
\end{align*}
Then there exist vectors $s_j\in\RR^r$, $j=1,2,3$ and constants $C_i$, $i=1,\ldots,4$, such that for any $\xx\in\DD_+$,
\begin{align*}
f(\xx) & = \log\Delta_{s_1}(I-\xx)+\log\Delta_{s_2}(\xx)+\log\Delta_{s_3}(I-\xx)+C_1, \\
g(\xx) & = \log\Delta_{s_1}(I-\xx)+\log\Delta_{s_3}(\xx)+C_2, \\
h(\xx) & = \log\Delta_{s_1}(I-\xx)+\log\Delta_{s_2}(I-\xx)+\log\Delta_{s_3}(\xx)+C_3, \\
k(\xx) & = \log\Delta_{s_1}(I-\xx)+\log\Delta_{s_2}(\xx)+C_4,
\end{align*}
and $C_1+C_2 = C_3+C_4$.
\end{corollary}
Observe that $\kappa\log\det\xx$ is the only continuous $w_1$- and $w_2$-logarithmic function, therefore we obtain the following
\begin{corollary}
Assume that $f$, $g$, $h$ and $k$ are continuous real functions defined on $\DD_+$ and
\begin{align*}
f(\xx)+g(t_{I-\xx}\cdot\yy\cdot t_{I-\xx}^T)=h(\yy)+k((I-\yy)^{-1/2}\cdot\xx\cdot (I-\yy)^{-1/2}),\quad(\xx,\yy)\in\DD^+_0.
\end{align*}
Then there exist constants $\kappa_1$, $\kappa_2$, $C_i$, $i=1,\ldots,4$, and vector $s_3\in\RR^r$, such that for any $\xx\in\DD_+$,
\begin{align*}
f(\xx) & = \kappa_1\log\det(\ee-\xx)+\kappa_2\log\det\xx+\log\Delta_{s_3}(I-\xx)+C_1, \\
g(\xx) & = \kappa_1\log\det(\ee-\xx)+\log\Delta_{s_3}(\xx)+C_2, \\
h(\xx) & = \kappa_1\log\det(\ee-\xx)+\kappa_2\log\det(I-\xx)+\log\Delta_{s_3}(\xx)+C_3, \\
k(\xx) & = \kappa_1\log\det(\ee-\xx)+\kappa_2\log\det\xx+C_4,
\end{align*}
and $C_1+C_2 = C_3+C_4$.
\end{corollary}

For necessary introduction on triangular group on $\VV$ we refer to \citet{FaKo1994} (see also \citet{wC2013}). 

We will now give the solution to fundamental equation of information on symmetric cones for any two multiplication algorithms, but with additional assumption that any two unknown functions are invariant under the automorphism group $K$ ($f(k\xx)=f(\xx)$ for any $k\in K$ and $\xx\in\VV$). In this way we obtain the same form of the solution, regardless of the choice of $\ww$ and $\widetilde{\ww}$. 
\begin{lemma}
Let $h_1$ and $h_2$ be two continuous $\ww$-logarithmic functions. If function $f(\xx)=h_1(\xx)+h_2(\ee-\xx)$, $\xx\in\DD$ is $K$-invariant, then there exist real constants $\kappa_1$ and $\kappa_2$ such that $f(\xx)=\kappa_1\log\det\xx+\kappa_2\log\det(\ee-\xx)$.
\end{lemma}
\begin{proof}
According to Theorem \ref{XXX} it is enough to show that both $h_1$ and $h_2$ are $K$-invariant. 
The following equality holds for any $k\in K$ and $\xx\in\DD$,
\begin{align*}
h_1(\alpha\xx)+h_2(\ee-\alpha\xx)=f(\alpha\xx)=f(\alpha k\xx)=h_1(\alpha k\xx)+h_2(\ee-\alpha k\xx)
\end{align*}
where $\alpha\in(0,1]$.
But we have $h_1(\alpha \xx)=h_1(\xx)+h_1(\alpha\ee)$ and $k\ee=\ee$, thus
\begin{align*}
h_1(\xx)+h_2(\ee-\alpha\xx)=h_1(k\xx)+h_2(\ee-\alpha k\xx)
\end{align*}
holds for any $\alpha\in(0,1]$. Passing to the limit as $\alpha\to0$ ($h_2$ is continuous and $h_2(\ee)=0$) implies that $h_1$ is $K$-invariant, and so is $h_2$.
\end{proof}

Following Corollary is a consequence of the above Lemma and Theorem \ref{inform}.
\begin{corollary}\label{cor2}
Assume that $f$, $g$, $h$ and $k$ are real functions defined on $\DD$ and \eqref{mainB} holds for multiplication algorithms $\ww=\gw^{-1}$ and $\widetilde{\ww}=\widetilde{\gw}^{-1}$ that satisfy conditions $A-C$. If any two unknown functions are $K$-invariant, then there exist constants $\kappa_j$, $j=1,2,3$ and $C_i$, $i=1,\ldots,4$, such that for any $\xx\in\DD$ we have \eqref{tryk}.
\end{corollary}

\bibliographystyle{plainnat}

\bibliography{Bibl}

\end{document}